\documentclass{amsart}

\title{Recovering a cohomological Mackey Functor from its restriction to Sylow subgroups}
\author{Vigleik Angeltveit}
\address{Mathematical Sciences Institute \\
Australian National University \\
Canberra, ACT 0200 \\
Australia}

\usepackage{amsxtra}
\usepackage{amsfonts}
\usepackage[latin1]{inputenc}
\usepackage{graphicx}
\usepackage{amsmath,amssymb,latexsym,amsthm,mathrsfs}
\usepackage[all]{xy}
\usepackage{pstricks}
\usepackage{verbatim}
\newtheorem{theorem}{Theorem}[section]
\newtheorem{thm}[theorem]{Theorem}
\newtheorem{lemma}[theorem]{Lemma}

\newtheorem{prop}[theorem]{Proposition}

\theoremstyle{definition}
\newtheorem{defn}[theorem]{Definition}
\newtheorem{remark}[theorem]{Remark}
\newtheorem{example}[theorem]{Example}

\makeatletter
\let\c@equation\c@theorem
\makeatother
\numberwithin{equation}{section}

\newcommand{\cA}{\mathcal{A}}              \newcommand{\cO}{\mathcal{O}}

  \newcommand{\bC}{\mathbb{C}}   \newcommand{\bF}{\mathbb{F}}          
 \newcommand{\bR}{\mathbb{R}}        \newcommand{\bZ}{\mathbb{Z}}

\newcommand{\sma}{\wedge} %smash product
\newcommand{\xto}{\xrightarrow}
\newcommand{\xfrom}{\xleftarrow}

\newcommand{\wh}{\widehat}
\newcommand{\im}{\textnormal{Im}}

\newcommand{\un}{\underline}

\newcommand{\tr}{\textnormal{Tr}}

\begin{document}

\begin{abstract}
We explain how to recover the top level of a cohomological $G$-Mackey functor $\un{M}$ from the restriction of $\un{M}$ to each of the Sylow subgroups of $G$.

As an application, we compute the Mackey functor valued $G$-equivariant homology groups of a point with constant $\bZ$-coefficients when $G$ has order $pq$ for odd primes $p < q$. We also indicate how the calculation goes for $G=A_4$.
\end{abstract}

\maketitle

\section{Introduction}
Let $G$ be a finite group and let $\un{M}$ be a $G$-Mackey functor. Then we might ask how much of $\un{M}$ is determined by its restriction to proper subgroups of $G$. For example, the relations between transfer and restriction maps in $\un{M}$ often force $\un{M}(G/G)$ to be non-zero.

In general we have no hope of recovering $\un{M}(G/G)$ from the rest of $\un{M}$, but if we assume that $G$ is not a $p$-group and that $\un{M}$ is \emph{cohomological} then this becomes possible:

\begin{thm} \label{t:main}
Let $G$ be a finite group and let $P_1,\ldots,P_n$ be a choice of Sylow subgroups, one for each prime dividing $|G|$. Then a cohomological $G$-Mackey functor $\un{M}$ is determined, up to isomorphism, by its restriction $\downarrow^G_{P_i} \un{M}$ for each $i=1,\ldots,n$ together with the action of $W_G(H) = N_G(H)/H$ on $( \downarrow^G_{P_i} \un{M} ) (P_i/H)$ for each $H \leq P_i$.
\end{thm}

This means that cohomological $G$-Mackey functors are not much more complicated than cohomological $P_i$-Mackey functors for $i=1,\ldots,n$.

As a specific example, suppose $G$ is the non-abelian group of order $pq$ for primes $p < q$ with $p \mid q-1$. Then a $G$-Mackey functor is given by a diagram
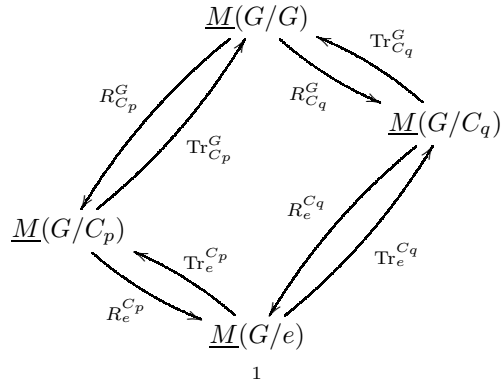
\[
 \xymatrix{ & \un{M}(G/G) \ar@/_/[rd]_-{R^G_{C_q}} \ar@/_/[ldd]_-{R^G_{C_p}} & \\ & & \un{M}(G/C_q) \ar@/_/[ul]_-{\tr_{C_q}^G} \ar@/_/[ldd]_-{R^{C_q}_e} \\ \un{M}(G/C_p) \ar@/_/[uur]_-{\tr_{C_p}^G} \ar@/_/[rd]_-{R^{C_p}_e} & & \\ & \un{M}(G/e) \ar@/_/[ruu]_-{\tr_e^{C_q}} \ar@/_/[lu]_-{\tr_e^{C_p}} & }
\]
together with a $G$-action on $\un{M}(G/e)$ and a $C_p$-action on $\un{M}(G/C_q)$. (There are some compatibility conditions.)

Now suppose we are given such a diagram, but without $\un{M}(G/G)$. Then, if we assume that $\un{M}$ is cohomological, we can recover $\un{M}(G/G)$ from the rest of the diagram as
\[
 \un{M}(G/G) = \un{M}(G/C_p) \oplus \un{M}(G/C_q)/\sim,
\]
where $\sim$ is defined in Theorem \ref{t:recover} below.

As an application, we compute the Mackey functor valued $G$-equivariant homology groups of a point with constant $\bZ$-coefficients when $G$ is the non-abelian group of order $pq$ for primes $p < q$ with $p \mid q-1$. For comparison we do the same for the abelian group of order $pq$.

We also indicate how the calculation goes for $G=A_4$, although we will not give a closed form answer as the calculation for $G=C_2 \times C_2$ is already quite complicated.

Various other groups have been considered in the literature, often with a focus on the slice spectral sequence in addition to the homology of a point. The cyclic group $C_2$ was considered by Dugger in \cite{Du06}, and $C_8$ was instrumental in the Hill-Hopkins-Ravenel solution to the Kervaire Invariant One problem \cite{HHR}. Yarnall \cite{Ya17} computed the slice spectral sequence for $\Sigma^n H_{C_{p^k}} \un{\bZ}$ for $p$ odd and $n>0$, and Hill, Hopkins and Ravenel \cite{HHR17} computed the slice spectral sequence for $\Sigma^{m\lambda} H_{C_{p^k}} \un{\bZ}$ for $p$ odd and $m > 0$. Hill and Yarnall \cite{HiYa18} have an abstract description of the slices of a $C_p$-spectrum. Guillou and Yarnall \cite{GuYa} have a complete description of the slice spectral sequence for $\Sigma^n H_{C_2 \times C_2} \un{\bF_2}$ for $n > 0$, and Ellis-Bloor \cite{EB20} gives a nice calculation of the $C_2 \times C_2$-equivariant homology groups of an arbitrary virtual representation sphere. Zou \cite{Zo18} has discussed the case of the dihedral group $D_{2p}$ of order $2p$ for an odd prime $p$. Finally, the author has a complete description of the homology of a point with coefficients in an arbitrary Mackey functor and the slice spectral sequence for an arbitrary suspension of a the Eilenberg-MacLane spectrum of an arbitrary Mackey functor for $G=C_p$ \cite{An21B}.

\subsection{Organization}
We start in Section \ref{s:cohomMF} by recalling the definition of a Mackey functor and of a cohomological Mackey functor before proving the main theorem in Section \ref{s:recover}.

The rest of the paper is dedicated to examples, with Sections \ref{s:repG} through \ref{s:Ghomologyofpoint} dedicated to the non-abelian group of order $pq$ for odd primes $p < q$ with $p \mid q-1$. For the reader's convenience we recally some facts about the real representations of $G$ in Section \ref{s:repG}. In Section \ref{s:GMackeyex} we introduce the $G$-Mackey functors that show up in the calculation of the $G$-equivariant homology of a point, and in Section \ref{s:Ghomologyofpoint} we actually compute the $G$-equivariant homology of a point.

In Section \ref{s:abgroupoforderpq} we do the same for the \emph{abelian} group of order $pq$ and in Section \ref{s:A4} we briefly indicate how the calculation goes for the alternating group $A_4$.

\section{Cohomological Mackey functors} \label{s:cohomMF}
Recall, e.g.\ from \cite{We00}, that a $G$-Mackey functor consists of an abelian group $\un{M}(G/H)$ for each subgroup $H \leq G$, together with the following structure maps:
\begin{enumerate}
 \item For $K \leq H \leq G$, a restriction map
 \[
  R^H_K : \un{M}(G/H) \to \un{M}(G/K).
 \]

 \item For $K \leq H \leq G$, a transfer map
 \[
  \tr_K^H : \un{M}(G/K) \to \un{M}(G/H).
 \]

 \item For each $g \in G$ and $H \leq G$, an action map
 \[
  c_g : \un{M}(G/H) \to \un{M}(G/H') \qquad \textnormal{for }H' = gHg^{-1}
 \]
\end{enumerate}

These should satisfy various identities. Rather than listing them here, we note that they are most easily discovered by thinking of $\un{M}$ as an additive functor from the category of spans of finite $G$-sets to abelian groups. The restriction map $R^H_K$ is obtained by evaluating $\un{M}$ on the span $G/H \leftarrow G/K \xto{=} G/K$ and the transfer map $\tr_K^H$ is obtained by evaluating $\un{M}$ on the span $G/K \xfrom{=} G/K \to G/H$. The composition law for spans is given by pullback, and this tells us how to rewrite $R^H_{K'} \circ \tr_K^H$ as a sum $\sum\limits_i \tr_{L_i'}^{K'} \circ c_{g_i} \circ R^K_{L_i}$.

We think of the map $c_g : \un{M}(G/H) \to \un{M}(G/H')$ as right multiplication by $g^{-1}$: 
\[
 c_g(aH) = aHg^{-1} = ag^{-1} gHg^{-1} = ag^{-1} H'.
\] 
(We have to use right multiplication to get a $G$-equivariant map of left $G$-sets.) It is an isomorphism, so it suffices to specify $\un{M}(G/H)$ for one representative of each conjugacy class of subgroups. This is clear from thinking about finite $G$-sets, as $G/H$ and $G/H'$ are isomorphic. Moreover, $c_g = id$ for $g \in H$, so we have an induced action of $W_G(H) = N_G(H)/H$ on each $\un{M}(G/H)$.

\begin{defn}
A Mackey functor $\un{M}$ is \emph{cohomological} if $\tr_K^H R^H_K$ is multiplication by the index $[K:H]$ for all $K \leq H \leq G$.
\end{defn}

\begin{example}
For any finite group $G$ and any abelian group $C$, the constant Mackey functor $\un{C}$ with $\un{C}(G/H) = C$ for all $H \leq G$, each $R^H_K$ being the identity map, and each $\tr_K^H$ being multiplication by $[K:H]$, is a cohomological Mackey functor.

Similarly, $\un{C}^*$ with each $R^H_K$ being multiplication by $[K:H]$ and each $\tr_K^H$ being the identity, is a cohomological Mackey functor.
\end{example}

\begin{example}
If $C$ is a $\bZ[G]$-module then the Mackey functor $F(C)$ defined by $F(C)(G/H) = C^H$, the $H$-fixed points of $C$, is a cohomological Mackey functor. In this case the restriction maps are given by inclusion of fixed points, while the transfer maps are given by a sum of group actions in the usual way.

Similarly, the Mackey functor $\cO(C)$ defined by $\cO(C)(G/H) = C_H$, the $H$-invariants (or orbits) of $C$, is a cohomological Mackey functor. In this case the transfer maps are given by projection, while the restriction maps are given by a sum of group actions.
\end{example}

\begin{example}
The Burnside Mackey functor $\un{\cA}$, with $\un{\cA}(G/H) = A(H)$ the Burnside ring of $H$ for each $H \leq G$, is not cohomological unless $G=e$. In particular, $\tr_e^G R^G_e([G/G]) = [G/e]$, which is not $|G|[G/G]$.
\end{example}

\begin{prop}
Let $X$ be a $G$-CW spectrum. Then the Mackey-functor valued homology groups of $X$ with coefficients in a cohomological Mackey functor are again cohomological.
\end{prop}

See \cite[Example 3.1.4]{Zo18} for a proof in the specal case $\un{M}=H\un{\bZ}$, $X=S^V$, and $G=D_{2p}$. The general case is probably known, but we include a proof sketch anyway.

\begin{proof}[Proof sketch]
% This is probably known, but we include a sketch proof in the general case.
If $X$ has $n$-cells $\coprod D^n \times G/H_i$ then $\un{H}_n(X; \un{M})$ can be computed using a chain complex of Mackey functors with $\un{C}_n = \big( G/K \mapsto \un{M}(\coprod G/H_i \times G/K) \big)$ and $\un{C}_n$ is again a cohomological Mackey functor. Moreover, kernels and cokernels of cohomological Mackey functors are again cohomological, so the result follows.
\end{proof}

\section{Recovering a Mackey functor from its restriction to Sylow subgroups} \label{s:recover}
As in the introduction we fix a set $P_1,\ldots,P_n$ of Sylow subgroups, one for each prime dividing $|G|$.

\begin{thm} \label{t:recover}
Given a cohomological Mackey functor $\un{M}$ we have
\[
 \un{M}(G/G) \cong \bigoplus_{i=1}^n \un{M}(G/P_i)/\!\sim
\]
where $\sim$ is generated by
\[
 \tr_H^{P_i}(x) \sim \tr_{H'}^{P_j} (c_g x)
\]
for $x \in \un{M}(G/H)$ and $g \in G$ with $H \leq P_i$ and $H'=gHg^{-1} \leq P_j$.
\end{thm}

In other words, we can recover the top level of $\un{M}$ from the value of $\un{M}$ on each of the Sylow subgroups. Similarly, we can recover $\un{M}(G/K)$ for all $K \leq G$ as well as the remaining restriction and transfer maps. Hence Theorem \ref{t:main} follows from Theorem \ref{t:recover}.

\begin{remark}
The equivalence relation in Theorem \ref{t:recover} includes the case $H=P_i$, with $\tr_{P_i}^{P_i}$ the identity. Since $P_i$ is not subconjugate to any other $P_j$ we must have $i=j$ and $g \in N_G(P_i)$, and since $P_i$ acts trivially on $\un{M}(G/P_i)$ we end up quotienting out by $W_G(P_i) = N_H(P_i)/P_i$.

Similarly, if $H \neq e$ then $H$ can be subconjugate to at most one $P_i$ so we must have $i=j$ in $\tr_H^{P_i}(x) \sim \tr_{H'}^{P_j}(x)$. But if $H=e$ we get $\tr_e^{P_i}(x) \sim \tr_e^{P_j}(c_g x)$ with no conditions on $i$, $j$ or $g$.
\end{remark}

\begin{proof}
We choose integers $r_1,\ldots,r_n$ with $\sum r_i [G:P_i] = 1$ and consider the map
\[
 (r_1 R^G_{P_1},\ldots,r_nR^G_{P_n}) : \un{M}(G/G) \to \bigoplus \un{M}(G/P_i).
\]
Let $\Phi : \un{M}(G/G) \to \bigoplus \un{M}(G/P_i)/\sim$ denote the induced map to $\bigoplus \un{M}(G/P_i)/\sim$.

In the other direction, we consider the map
\[
 (\tr_{P_1}^G,\ldots,\tr_{P_n}^G) : \bigoplus \un{M}(G/P_i) \to \un{M}(G/G).
\]
From the axioms for a $G$-Mackey functor, this map is constant on equivalence classes so we get an induced map $\Psi : \bigoplus \un{M}(G/P_i)/\!\sim\, \to \un{M}(G/G)$.

Using that $\un{M}$ is cohomological we find that
\[
 \Psi \circ \Phi (x) = \sum r_i \tr_{P_i}^G \circ R^G_{P_i} (x) = \sum r_i [G:P_i] x = x,
\]
so $\Psi \circ \Phi = id_{\un{M}(G/G)}$.

Conversely, if we start with a class in $\bigoplus \un{M}(G/P_i)$ represented by $x \in \un{M}(G/P_k)$ we find that
\[
 \Phi \circ \Psi (x) = \sum_i r_i R^G_{P_i} \tr_{P_k}^G (x).
\]
The actual formula for $R^G_{P_i} \tr_{P_k}^G$ might be complicated, but we claim that we have
\[
 R^G_{P_i} \tr_{P_k}^G (x) \sim [G : P_i] x.
\]
It then follows that $\Phi \circ \Psi (x) = x$, and we are done.

To see this, we note that we can compute $R^G_{P_i} \circ \tr_{P_k}^G$ by decomposing $G/P_k \times G/P_i$ into orbits, and we get some formula
\[
 R^G_{P_i} \tr_{P_k}^G = \sum_j \tr_{H_j'}^{P_i} \circ c_{g_j} \circ R^{P_k}_{H_j}.
\]
If $i=k$ then different subgroups $H_j$ are possible, and if $i \neq k$ then $H_j = e$.

Because $|G/P_k \times G/P_i| = [G : P_k] [G : P_i]$, we must have $\sum\limits_j [P_k : H_j] = [G : P_i]$.

Now we use that
\[
 \tr_{H_j'}^{P_i} \circ c_{g_j} \circ  R^{P_k}_{H_j} (x) \sim \tr_{H_j}^{P_k} R^{P_k}_{H_j} (x).
\]
Since $\un{M}$ is cohomological the right hand side is $[P_k : H_j] x$, and we conclude that $R^G_{P_i} \tr_{P_k}^G (x) \sim [G:P_i] x$. Hence $\Phi \circ \Psi (x) \sim \sum\limits_i r_i [G:P_i] x = x$, and the result follows. 
\end{proof}

\section{Some representation theory for the nonabelian group of order $pq$} \label{s:repG}
All the material in this section is well known, but we include it for the reader's convenience and to fix notation. Fix primes $p < q$ with $p \mid q-1$. Then there are exactly two groups of order $pq$ up to isomorphism (one if $p \nmid q-1$). There is the cyclic group $C_{pq} \cong C_p \times C_q$, and a unique non-abelian group $G = C_q \rtimes C_p$ defined by a map $C_p \to Aut(C_q)$ which we suppress from the notation. Here $C_q \triangleleft G$ is normal, while there are $q$ conjugate subgroups isomorphic to $C_p$.

We fix a presentation
\[
 G = \langle a, b \quad | \quad a^p, b^q, ba=ab^k \rangle.
\]

\subsection{Complex representations} \label{ss:Crep}
The irreducible complex representations of $C_p = \langle a \quad | \quad a^p \rangle$ are as follows: First we have the trivial representation $1$. Then we have the $1$-dimensional representations $V_i$ for $1 \leq i \leq p-1$, where $a$ acts on $V_i$ by rotation by $i \cdot \frac{2\pi}{p}$. We denote the corresponding representations of $C_q = \langle b \quad | \quad b^q \rangle$ by $V_i'$ for $1 \leq i \leq q-1$.

The irreducible complex representations of $G$ are as follows: First, we have the trivial representation $1$. Then we have the $1$-dimensional representations $V_i$ for $1 \leq i \leq p-1$, with $G$ acting through the quotient $G \to G/C_q \cong C_p$. (We use the same notation for the $G$-representation and the corresponding $C_p$-representation.) With our presentation, $a$ acts by rotation by $i \cdot \frac{2\pi}{p}$ while $b$ acts trivially. Finally, we have the $p$-dimensional representations $W_i$ for $i \in I$, where $I \subset \{1,\ldots,q-1\}$ is a subset of cardinality $\frac{q-1}{p}$. These can most easily be defined as $Ind_{C_q}^G V_i'$ for $i \in I$.

The representation $W_i$ can be described as $\bC^p$, with $C_p$ cyclically permuting the coordinates and $Res^G_{C_q} W_i = V_i' \oplus V_{ki}' \oplus \ldots \oplus V_{k^{p-1} i}'$. The set $I$ can be described as follows. We consider the action of $C_p$ on $\{1,\ldots,q-1\}$ where $a$ acts by multiplication by $k$. This partitions $\{1,\ldots,q-1\}$ into $\frac{q-1}{p}$ subsets, and then we pick one $i$ from each subset.

\subsection{Real representations}
Except for the trivial representation, all the irreducible representations of $G$ are complex (as opposed to real or quaternionic), and it follows that the irreducible real representations of $G$ are as follows: First, we have the trivial representation $1$. Then we have the $2$-dimensional representations $V_i$ for $1 \leq i \leq \frac{p-1}{2}$, with $V_i \cong V_{p-i}$. Finally, we have the $2p$-dimensional representations $W_j$ for $j \in J$ for a set $J$ of cardinality $\frac{q-1}{2p}$.

\begin{example}
Let $p=3$ and $q=13$. Then we can choose the presentation $\langle a, b \quad  | \quad a^3, b^{13}, ba=ab^3 \rangle$ for $G$. The irreducible complex representations of $G$ are $1$, $V_1$, $V_2$, $W_1$, $W_2$, $W_4$ and $W_7$. We have
\begin{eqnarray*}
R^G_{C_{13}}(W_1) & \cong & V_1' \oplus V_3' \oplus V_9' \\
R^G_{C_{13}}(W_2) & \cong & V_2' \oplus V_6' \oplus V_5' \\
R^G_{C_{13}}(W_4) & \cong & V_4' \oplus V_{12}' \oplus V_{10}' \\
R^G_{C_{13}}(W_7) & \cong & V_7' \oplus V_8' \oplus V_{11}'
\end{eqnarray*}

When considering real representations instead we have $V_1 \cong V_2$, $W_1 \cong W_4$, and $W_2 \cong W_7$. So in this case we can take $I = \{1,2,4,7\}$ and $J = \{1,2\}$, and it follows that $1$, $V_1$, $W_1$, $W_2$ is a complete list of irreducible real representations of $G$. Other choices of $I$ and $J$ are possible.
\end{example}

\section{Some cohomological $G$-Mackey functors} \label{s:GMackeyex}
In this section $G$ still denotes the nonabelian group of order $pq$. We will represent a $G$-Mackey functor by a diagram
\[
 \xymatrix{ & \un{M}(G/G) \ar@/_/[rd]_-{R^G_{C_q}} \ar@/_/[ldd]_-{R^G_{C_p}} & \\ & & \un{M}(G/C_q) \ar@/_/[ul]_-{\tr_{C_q}^G} \ar@/_/[ldd]_-{R^{C_q}_e} \\ \un{M}(G/C_p) \ar@/_/[uur]_-{\tr_{C_p}^G} \ar@/_/[rd]_-{R^{C_p}_e} & & \\ & \un{M}(G/e) \ar@/_/[ruu]_-{\tr_e^{C_q}} \ar@/_/[lu]_-{\tr_e^{C_p}} & }
\]
Here we omit the additional subgroups that are conjugate to $C_p$ from the diagram. This is not a serious omission, since if $C_p'$ is conjugate to $C_p$ then $\un{M}(G/C_p')$ is isomorphic to $\un{M}(G/C_p)$. If the $G$-action on $\un{M}(G/e)$ or the $C_p$-action on $\un{M}(G/C_q)$ is nontrivial we will indicate that in the notation. See Example \ref{ex:Cqhat} below.

We will also need to know how to compose a transfer map with a restriction map. In particular we find the following:
\begin{eqnarray*}
 R^G_{C_p} \circ \tr_{C_p}^G (x) & = & x + \sum_{i \in I} \tr_e^{C_p} \big( b^i R^{C_p}_e(x) \big) \\
 R^G_{C_q} \circ \tr_{C_p}^G (x) & = & \tr_e^{C_q} \circ R^{C_p}_e (x) \\
 R^G_{C_p} \circ \tr_{C_q}^G (x) & = & \tr_e^{C_p} \circ R^{C_q}_e (x) \\
 R^G_{C_q} \circ \tr_{C_q}^G (x) & = & \sum_{i=0}^{p-1} a^i x.
\end{eqnarray*}
Here $I$ is as in the definition of the irreducible complex $G$-representations from Subsection \ref{ss:Crep}, and $a$ and $b$ are the same as in the presentation of $G$ given at the start of Section \ref{s:repG}.

We give some examples of cohomological Mackey functors for $G$:

\begin{example}
We have multiple versions of the constant Mackey functor $\un{\bZ}$.
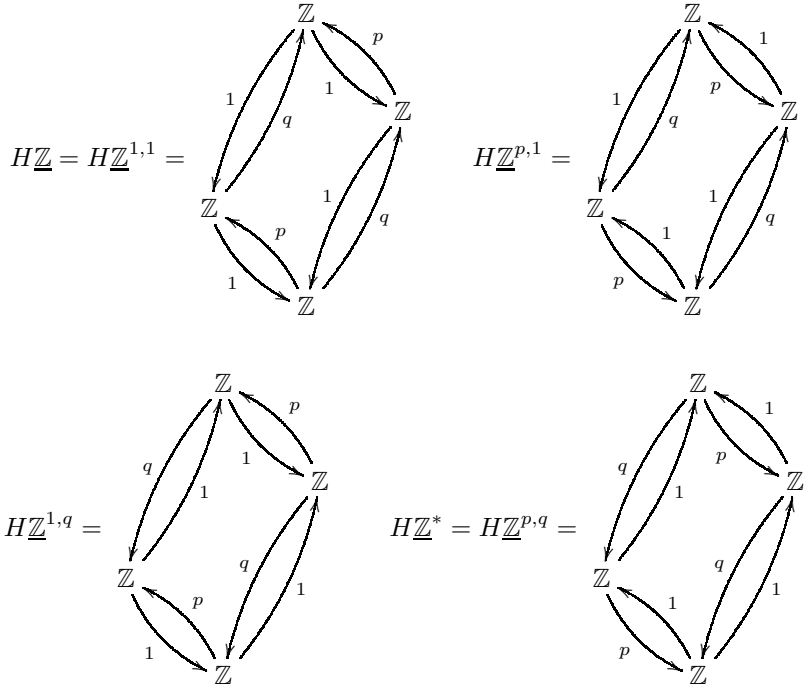
\[
 H\un{\bZ} = H\un{\bZ}^{1,1} = \vcenter{\hbox{\xymatrix{ & \bZ \ar@/_/[rd]_-1 \ar@/_/[ldd]_-1 & \\ & & \bZ \ar@/_/[ul]_-p \ar@/_/[ldd]_-1 \\ \bZ \ar@/_/[uur]_-q \ar@/_/[rd]_-1 & & \\ & \bZ \ar@/_/[ruu]_-q \ar@/_/[lu]_-p & } }} \qquad
 H\un{\bZ}^{p,1} = \vcenter{\hbox{\xymatrix{ & \bZ \ar@/_/[rd]_-p \ar@/_/[ldd]_-1 & \\ & & \bZ \ar@/_/[ul]_-1 \ar@/_/[ldd]_-1 \\ \bZ \ar@/_/[uur]_-q \ar@/_/[rd]_-p & & \\ & \bZ \ar@/_/[ruu]_-q \ar@/_/[lu]_-1 & } }}
\]

\[
 H\un{\bZ}^{1,q} = \vcenter{\hbox{\xymatrix{ & \bZ \ar@/_/[rd]_-1 \ar@/_/[ldd]_-q & \\ & & \bZ \ar@/_/[ul]_-p \ar@/_/[ldd]_-q \\ \bZ \ar@/_/[uur]_-1 \ar@/_/[rd]_-1 & & \\ & \bZ \ar@/_/[ruu]_-1 \ar@/_/[lu]_-p & } }} \qquad
 H \un{\bZ}^* = H\un{\bZ}^{p,q} = \vcenter{\hbox{\xymatrix{ & \bZ \ar@/_/[rd]_-p \ar@/_/[ldd]_-q & \\ & & \bZ \ar@/_/[ul]_-1 \ar@/_/[ldd]_-q \\ \bZ \ar@/_/[uur]_-1 \ar@/_/[rd]_-p & & \\ & \bZ \ar@/_/[ruu]_-1 \ar@/_/[lu]_-1 & } }}
\]

In each case the $G$-action on the $\bZ$ in position $G/e$ and the $C_p$-action on the $\bZ$ in position $G/C_q$ are necessarily trivial.
\end{example}

\begin{example}
We have a Mackey functor
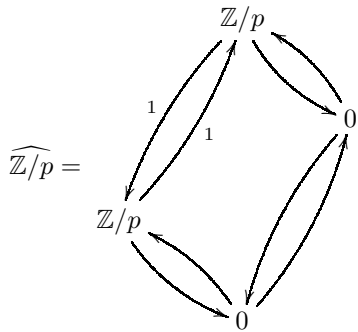
\[
 \wh{\bZ/p} = \vcenter{\hbox{\xymatrix{ & \bZ/p \ar@/_/[rd] \ar@/_/[ldd]_-1 & \\ & & 0 \ar@/_/[ul] \ar@/_/[ldd] \\ \bZ/p \ar@/_/[uur]_-1 \ar@/_/[rd] & & \\ & 0 \ar@/_/[ruu] \ar@/_/[lu] & } }}
\]
Note that the restriction and transfer maps are both $1$, rather than $1$ and $q$. It has to be like this, by the relations spelled out above.
\end{example}

\begin{example} \label{ex:Cqhat}
Let $\bZ/q(e)$ denote $\bZ/q$ with the generator of $C_p$ acting by multiplication by $e$. This only makes sense if $e^p \equiv 1 \mod q$, but that includes $e=k, k^2,\ldots, k^{p-1}$. (Here $k$ is the same as the $k$ in our presentation of $G$.) If $e=1$ we omit it from the notation. Then we have the Mackey functors
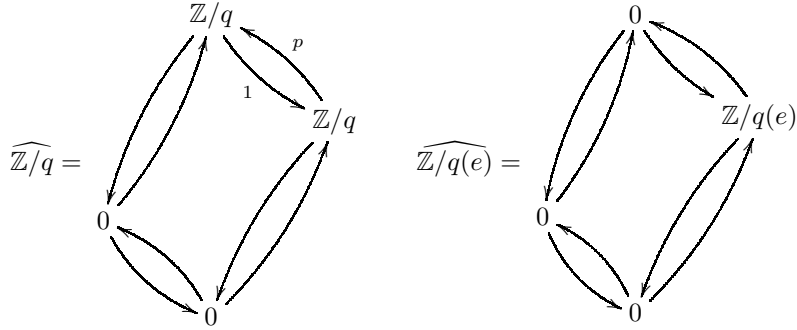
\[
 \wh{\bZ/q} = \vcenter{\hbox{\xymatrix{ & \bZ/q \ar@/_/[rd]_-1 \ar@/_/[ldd] & \\ & & \bZ/q \ar@/_/[ul]_-p \ar@/_/[ldd] \\ 0 \ar@/_/[uur] \ar@/_/[rd] & & \\ & 0 \ar@/_/[ruu] \ar@/_/[lu] & } }} \qquad
 \wh{\bZ/q(e)} = \vcenter{\hbox{\xymatrix{ & 0 \ar@/_/[rd] \ar@/_/[ldd] & \\ & & \bZ/q(e) \ar@/_/[ul] \ar@/_/[ldd] \\ 0 \ar@/_/[uur] \ar@/_/[rd] & & \\ & 0 \ar@/_/[ruu] \ar@/_/[lu] & } }}
\]
We note that the Mackey functor relation
\[
 R^G_{C_q} \tr_{C_q}^G (y) = \sum_{i=0}^{p-1} a^i y
\]
is satisfied for $\wh{\bZ/p(e)}$ because $y + ey + \ldots + e^{p-1} y = 0$.
\end{example}

\begin{lemma} \label{l:boxofCq}
With notation as above we have
\[
 \wh{\bZ/q(e_1)} \Box \wh{\bZ/q(e_2)} \cong \wh{\bZ/q(e_1 e_2)}.
\]
\end{lemma}

\begin{proof}
In general, given two $G$-Mackey functors $P$ and $Q$, and a subgroup $H < G$, we have
\[
 P \Box Q (G/H) \cong \Big( P(G/H) \otimes Q(G/H) \oplus \im(\tr) \Big)/\sim,
\]
where $\sim$ is the equivalence relation given by Frobenius Reciprocity: $x \otimes \tr_K^H(y) \sim \tr_K^H(R^H_K(x) \otimes y)$ and $\tr_K^H(x) \otimes y \sim \tr_K^H(x \otimes R^H_K(y))$. In the case at hand, we get
\[
 \wh{\bZ/q(e_1)} \Box \wh{\bZ/q(e_2)} (G/C_q) \cong \bZ/q(e_1) \otimes \bZ/q(e_2) \cong \bZ/q(e_1 e_2),
\]
so $\wh{\bZ/q(e_1)} \Box \wh{\bZ/q(e_2)}$ and $\wh{\bZ/q(e_1 e_2)}$ agree at level $G/C_q$.

It remains to consider the top level. If $e_1 \neq 1$ or $e_2 \neq 1$, the summand $\wh{\bZ/q(e_1)}(G/G) \otimes \wh{\bZ/q(e_2)}(G/G)$ is trivial and we are left with $\im(\tr_{C_q}^G)$. This is precisely $\bZ/q$ if $e_1 e_2 =1$, and $0$ if $e_1 e_2 \neq 1$.

If $e_1 = e_2 = 1$, we get $\bZ/q \oplus \bZ/q/\sim$, and in this case Frobenius Reciprocity imposes the appropriate relation and we are left with a single $\bZ/q$.
\end{proof}

\section{The $G$-equivariant homology of a point} \label{s:Ghomologyofpoint}
As above we let $G$ be the non-abelian group of order $pq$.

Let $V$ be a virtual $G$-representation. Then each $\un{H}_n(S^V; \un{\bZ})$ (also known as the homology of a point in degree $n-V$) is a cohomological Mackey functor. Hence we can use Theorem \ref{t:recover} to recover $\un{H}_n(S^V; \un{\bZ})(G/G)$ from the lower levels.

A virtual $G$-representation $V$ can be written as
\[
 V \cong t \bR \oplus \bigoplus_{i=1}^{(p-1)/2} r_i V_i \oplus \bigoplus_{j \in J} s_j W_j,
\]
but as it turns out the calculation only depends on $r = \sum r_i$ and $s = \sum s_j$.

\subsection{The restriction to $C_p$}
We write $\un{\bZ}$, $\un{\bZ}^*$ and $\wh{\bZ/p}$ for the restriction of the corresponding Mackey functors from Section \ref{s:GMackeyex} to $C_p$.

We have $R^G_{C_p}(V) \cong t \bR \oplus \bigoplus r_i V_i \oplus 2s \rho_{C_p}$. Here $\rho_{C_p} \cong 1 + \sum\limits_{i=1}^{(p-1)/2} V_i$, so $R^G_{C_p}(V) \cong (t+2s) \bR \oplus \bigoplus\limits_{i=1}^{(p-1)/2} (r_i+2s) V_i$. The calculation is identical to the one for $\overline{V} = (t+2s) \bR \oplus (r+(p-1)s) V_1$.

Hence the homology depends on the value of $r + (p-1)s$.

If $r + (p-1)s = 0$ then $V$ behaves like a trivial representation and the homology is a single copy of $\un{\bZ}$ in degreee $\dim_\bR(V) = t+2s$.

If $r + (p-1)s > 0$ then $V$ behaves like an actual representation and the homology is given by a $\wh{\bZ/p}$ in dimension $t+2s+2i$ for $i=0,\ldots,r + (p-1)s-1$ and a $\un{\bZ}$ in dimension $\dim_\bR(V) = t+2r+2ps$.

If $r + (p-1)s < 0$ then $V$ behaves like the negative of an actual representation and the homology is given by a $\un{\bZ}^*$ in dimension $\dim_\bR(V) = t+2r+2ps$ and a $\wh{\bZ/p}$ in dimension $t+2r-2i-1$ for $i=1,\ldots,-(r+(p-1)s+1)$. (If $r+(p-1)s = -1$ there are no $\wh{\bZ/p}$ summands.)

\subsection{The restriction to $C_q$}
We write $\un{\bZ}$, $\un{\bZ}^*$ and $\wh{\bZ/q}$ for the restriction of the corresponding Mackey functors from Section \ref{s:GMackeyex} to $C_q$. Note that this restriction functor forgets the $C_p$-action at level $G/C_q$, so any of the $G$-Mackey functors $\wh{\bZ/q(e)}$ restrict to $\wh{\bZ/q}$. We have $R^G_{C_q}(V) \cong (t+2r) \bR \oplus \bigoplus s_j R^G_{C_q}(W_j)$, where $R^G_{C_q}(W_j)$ is a direct sum of $p$ irreducible $2$-dimensional $C_q$-representations. The calculation is identical to the one for $\overline{V} = (t+2r) \bR \oplus ps V_1'$.

This time the calculation depends on the value of $s$.

If $s=0$ then $V$ behaves like a trivial representation and the homology is a single copy of $\un{\bZ}$ in degree $\dim_\bR(V) = t+2r$.

If $s > 0$ then $V$ behaves like an actual representation and we get a $\wh{\bZ/q}$ in dimension $t+2r+2i$ for $i=0,\ldots,ps-1$ and a $\un{\bZ}$ in dimension $\dim_\bR(V) = t+2r+2ps$.

If $s < 0$ then $V$ behaves like the negative of an actual representation and the homology is given by a $\un{\bZ}^*$ in dimension $\dim_\bR(V) = t+2r+2ps$ and a $\wh{\bZ/q}$ in dimension $t+2r-2i-1$ for $i=1,\ldots,-(ps+1)$.

\subsection{The $C_p$-action on the restriction to $C_q$}
We need one more piece of information before we can use Theorem \ref{t:recover} to put the above calculations together to get the relevant $G$-Mackey functors, namely the $C_p$-action on each $\un{H}_n(S^V; \un{\bZ})(G/C_q)$.

The difficult part is computing the $C_p$-action on $H_n(S^{W_j}; \un{\bZ})(G/C_q)$. This is non-trivial because the usual $C_q$-CW structure on $R^G_{C_q}(S^{W_j})$ is not compatible with the $C_p$-action. We will focus on $S^{W}$ for $W=W_1$. The case of $S^{W_j}$ for $j \in J$ is similar. The usual $C_q$-CW structure on
\[
 R^G_{C_q} S^{W} = S^{R^G_{C_q}(W)} \cong S^{V_1'} \sma S^{V_k'} \sma \ldots \sma S^{V_{k^{p-1}}'}
\]
has a single $C_q/C_q$-cell in dimension $0$ and a single $C_q/e$-cell in dimension $1$ through $2p$. Let us write $C_q = \{1,t,\ldots,t^{q-1}\}$. We get the following chain complex of $\bZ[C_q]$-modules:
\[
 \bZ \xfrom{\nabla} \bZ[C_q] \xfrom{1-t} \bZ[C_q] \xfrom{\sum t^j} \bZ[C_q] \xfrom{1-t^k} \bZ[C_q] \leftarrow \ldots \leftarrow \bZ[C_q] \xfrom{1-t^{k^{p-1}}} \bZ[C_q].
\]
The attaching map from degree $2i$ to degree $2i-1$ is $1-t^{k^{i-1}}$ because the $i$'th smash factor is $S^{V_{k^{i-1}}'}$.

We can compute the homology at level $C_q/e$ by taking the homology groups of this chain complex, and we can compute the homology at level $C_q/C_q$ by first taking the $C_q$-fixed points and then taking the homology groups of the corresponding chain complex. The restriction map is induced by the inclusion of fixed points, and the transfer map is induced by $x \mapsto \sum t^j x$. On $C_q$-fixed points, the fold map $\nabla$ becomes multiplication by $q$, each map labelled $1-t^{k^{i-1}}$ becomes trivial, and each map labelled $\sum t^j$ becomes multiplication by $q$. Hence we recover the calculation of the $C_q$-Mackey functor $\un{H}_*(S^{R^G_{C_q}(W)}; \un{\bZ})$.

There is an alternative $C_q$-CW structure on $S^{R^G_{C_q}(W)}$ which is compatible with the $C_p$-action in the sense that the $C_p$-action sends cells to cells. Each $S^{V_{k^i}'}$ has a $C_q$-CW structure with a $C_q/C_q$-cell in dimension $0$, a $C_q/e$-cell in dimension $1$, and a $C_q/e$-cell in dimension $2$. The corresponding chain complex of $\bZ[C_q]$-modules looks as follows:
\[
 D_i = \bZ \xfrom{\nabla} \bZ[C_q] \xfrom{1-t^{k^i}} \bZ[C_q].
\]
We get a $C_q$-CW structure on $S^{V_1' \oplus V_k' \oplus \ldots \oplus V_{k^{p-1}}'} \cong S^{V_1'} \sma S^{V_k'} \sma \ldots \sma S^{V_{k^{p-1}}'}$ by taking the ``smash product'' of the $C_q$-CW structure on each smash factor, and this corresponds to taking the tensor product of the $D_i$. Because each $S^{V_{k^i}'}$ has a second $C_q/C_q$-cell in dimension zero which we are ignoring because we are computing reduced homology, the smash product contains all possible products of cells.

From this we get a $p$-dimensional chain complex of $\bZ[C_q]$-modules. The homology at $C_q/e$ is given by the homology of the associated total chain complex, and the homology at $C_q/C_q$ is given by the homology of the $C_q$-fixed points of the associated total chain complex.

To keep track of the smash factors, and the various directions in the $p$-dimensional chain complex, let $C_q^{(i)}$ denote the $i$'th copy of $C_q$ and let $t_i$ denote the generator of $C_q^{(i)}$. For example, consider $S^{V_1'} \sma S^{V_k'}$, with associated chain complex $D_0 \otimes D_1$. If we write $D_0$ vertically and $D_1$ horizontally this looks as follows:
\[ \xymatrix{
 \bZ & \bZ[C_q^{(1)}] \ar[l]_-{\nabla_1} & \bZ[C_q^{(1)}] \ar[l]_-{1-t_1^k} \\
 \bZ[C_q^{(0)}] \ar[u]^{\nabla_0} & \bZ[C_q^{(0)} \times C_q^{(1)}] \ar[l]_-{\nabla_1} \ar[u]^{\nabla_0} & \bZ[C_q^{(0)} \times C_q^{(1)}] \ar[l]_-{1-t_1^k} \ar[u]^{\nabla_0} \\
 \bZ[C_q^{(0)}] \ar[u]^{1-t_0} & \bZ[C_q^{(0)} \times C_q^{(1)}] \ar[l]_-{\nabla_1} \ar[u]^{1-t_0} & \bZ[C_q^{(0)} \times C_q^{(1)}] \ar[l]_-{1-t_1^k} \ar[u]^{1-t_0}
} \]
Here $\nabla_0$ sends $t_0$ to $1$ and $\nabla_1$ sends $t_1$ to $1$. We will not draw the whole $p$-dimensional chain complex, but it is similar. In particular it has the above double complex as a subcomplex, and that is all we need.

Because we have put two $C_q$-CW structures on the same $C_q$-space, we get the same Mackey functor valued homology groups from both calculations. But with this latter $C_q$-CW structure the $C_p$-action is given by cyclically permuting the smash factors, so there is an induced $C_p$-action on the chain complex and we can use that to compute the $C_p$-action on the homology groups.

If we take the $C_q$-fixed points of the above double complex we get the following:
\[ \xymatrix{
 \bZ & \bZ\{\sum t_1^i\} \ar[l]_-{\nabla_1} & \bZ\{\sum t_1^i\} \ar[l]_-{0} \\
 \bZ\{\sum t_0^i\} \ar[u]^{\nabla_0} & \bZ\{\sum t_0^i t_1^i,\ldots,\sum t_0^i t_1^{i+q-1}\} \ar[l]_-{\nabla_1} \ar[u]^{\nabla_0} & \bZ\{\sum t_0^i t_1^i,\ldots,\sum t_0^i t_1^{i+q-1}\} \ar[l]_-{1-t_1^k} \ar[u]^{\nabla_0} \\
 \bZ\{\sum t_0^i\} \ar[u]^0 & \bZ\{\sum t_0^i t_1^i,\ldots,\sum t_0^i t_1^{i+q-1}\} \ar[l]_-{\nabla_1} \ar[u]^{1-t_0} & \bZ\{\sum t_0^i t_1^i,\ldots,\sum t_0^i t_1^{i+q-1}\} \ar[l]_-{1-t_1^k} \ar[u]^{1-t_0}
} \]
Here all the sums are for $i$ from $0$ to $q-1$. We can use a spectral sequence argument to compute the homology groups of the associated double complex. If we first take homology in the vertical direction we are left with
\[ \xymatrix{
 \bZ/q & 0 & 0 \\
 0 & 0 & 0 \\
 \bZ & \bZ \ar[l]_-q & \bZ \ar[l]_-0,
} \]
so one generator of the $\bZ/q$ in degree $2$ is given by $\sum t_0^i$ in the lower left corner modulo filtration. Because $\sum t_0^i$ is indeed a cycle, $\sum t_0^i$ is indeed a generator of the $\bZ/q$ in degree $2$.

If instead we take homology in the horizontal direction we find that another generator of the $\bZ/q$ in degree $2$ is given by $\sum t_1^i$ in the upper right corner..

We can relate these two generators by noting that we have the following zig-zag:
\[ \xymatrix{
 & & \sum t_1^i \\
 & - \Big( \sum t_0^i t_1^i - \sum t_0^i t_1^{i+k} \Big) & \sum t_0^i t_1^i \ar@{|->}[u]_-{\nabla_0} \ar@{|->}[l]_-{-(1-t_1^k)} \\
 k \sum t_0^i & \sum t_0^i t_1^{i+1} + \ldots + \sum \sum t_0^i t_1^{i+k} \ar@{|->}[u]_{1-t_0} \ar@{|->}[l]_-{1-\nabla_1} &
} \]
We need to take the negative of the map $1-t_1^k$ because of the sign rules in the total complex of a double complex. The upshot is that $\sum t_1^i$ is homologous to $k \sum t_0^i$.

If we start with the generator $\sum t_0^i$ of the $\bZ/q$ in degree $2$ and act by the generator of $C_p$ we end up with the generator $\sum t_1^i$, and because this is homologous to $k$ times the generator we started with we can conclude that the generator of $C_p$ acts as multiplication by $k$.

\begin{thm} \label{t:Cpaction1}
Let $W$ be one of the $2p$-dimensional irreducible representations of $G$. Then the $C_p$-action on
\[
 \un{H}_{2i}(S^W; \un{\bZ})(G/C_q) \cong \un{H}_{2i}(S^{R^G_{C_q}(W)}; \un{\bZ})(C_q/C_q) \cong \bZ/q
\]
for $i=0,1,\ldots,p-1$ is given by multiplication by $k^i$.
\end{thm}

\begin{proof}
The above discussion proves this for $i=0$ and $i=1$. A similar argument works for any $i$, but we can also prove this by induction on $i$ using the induced maps
\[
 \un{H}_2(S^W; \un{\bZ}) \Box \un{H}_{2i}(S^W; \un{\bZ}) \to \un{H}_{2i+2}(S^W \sma S^W; \un{\bZ}) \leftarrow \un{H}_{2i+2}(S^W; \un{\bZ}).
\]
For $i$ and $i+1$ in the range we are considering, both maps are an isomorphism at level $G/C_q$. Both maps are also $C_p$-equivariant, so the result follows.
\end{proof}

We have a similar result for negative representations:

\begin{thm} \label{t:Cpaction2}
With notation as above, the $C_p$-action on
\[
 \un{H}_{-2i-1}(S^{-W}; \un{\bZ})(G/C_q) \cong \un{H}_{-2i-1}(S^{-R^G_{C_q}(W)}; \un{\bZ})(C_q/C_q) \cong \bZ/q
\]
for $i=1,\ldots,p-1$ is given by multiplication by $k^{-i}$.
\end{thm}

\begin{proof}
We dualise the argument from the previous theorem. The relevant double complex of $\bZ[C_q]$-modules for $S^{-V_1'} \sma S^{-V_k'}$ is\[ \xymatrix{
 \bZ \ar[r]^-{\Delta_1} \ar[d]_{\Delta_0} & \bZ[C_q^{(1)}] \ar[r]^-{1-t_1^k} \ar[d]_{\Delta_0} & \bZ[C_q^{(1)}] \ar[d]_{\Delta_0} \\
 \bZ[C_q^{(0)}] \ar[d]_{1-t_0} \ar[r]^-{\Delta_1} & \bZ[C_q^{(0)} \times C_q^{(1)}] \ar[r]^-{1-t_1^k} \ar[d]_{1-t_0} & \bZ[C_q^{(0)} \times C_q^{(1)}] \ar[d]_{1-t_0}  \\
 \bZ[C_q^{(0)}] \ar[r]^-{\Delta_1} & \bZ[C_q^{(0)} \times C_q^{(1)}] \ar[r]^-{1-t_1^k} & \bZ[C_q^{(0)} \times C_q^{(1)}]
} \]

After taking $C_q$-fixed points and taking homology we get a $\bZ/q$ in degree $-3$ and a $\bZ$ in degree $-4$. One generator is
\[
 \sum t_0^i t_1^i + \Big( \sum t_0^i t_1^{i+1} + \ldots + \sum t_0^i t_1^{i+k} \Big),
\]
where the first sum is from bidegree $(-2,-1)$ and the remaining sums are from bidegree $(-1,-2)$. Another generator is
\[
 \Big( \sum t_0^{i+k} t_1^i + \ldots + \sum t_0^{i+k\ell} t_1^i \Big) + \sum t_0^i t_1^i,
\]
where $\ell = k^{-1}$ and this time the first sums are from bidegree $(-2,-1)$ and the last sum is from bidegree $(-1,-2)$. The second generator is $\ell=k^{-1}$ times the first generator, and the result for $i=1$ follows in the same way as for Theorem \ref{t:Cpaction1}.

Once again we can either generalise the above argument to arbitrary $i$ or use the map
\[
 \un{H}_2(S^W; \un{\bZ}) \Box \un{H}_{-2(i+1)-1}(S^{-2W}; \un{\bZ}) \to \un{H}_{-2i-1}(S^{-W}; \un{\bZ}).
\]
to argue by induction.
\end{proof}

Because the action of $a$ on $S^{V_i}$ is homotopic to the identity, additional smash factors of the form $S^{V_i}$ will simply shift the homology groups at level $G/C_q$ without changing the $C_p$-action.

\begin{remark}
A $C_q$-CW structure on $S^{R^G_{C_q}(W)}$ which is compatible with the $C_p$-action is almost a $G$-CW structure on $S^W$. Presumably it is possible to compute the homology groups of a point using such a $G$-CW structure as well.
\end{remark}

We can now put this together to obtain the action of $C_p$ on $\un{H}_*(S^V; \un{\bZ})(G/C_q)$ for all $V$:

\begin{thm} \label{t:Cpaction3}
Let $G$ be the nonabelian group of order $pq$ for $p \mid q-1$, and let $V = t \bR \oplus \bigoplus\limits_{i=1}^{(p-1)/2} r_i V_i \oplus \bigoplus\limits_{j \in J} s_j W_j$. Let $r=\sum r_i$ and $s=\sum s_j$. In all cases the $C_p$-action on $\un{H}_{t+2r+2ps}(S^V; \un{\bZ})(G/C_q) \cong \bZ$ is trivial.
\begin{enumerate}
 \item If $s > 0$ then for $i=0,\ldots,ps-1$ the $C_p$-action on \[\un{H}_{t+2r+2i}(S^V; \un{\bZ})(G/C_q) \cong \bZ/q\] is given by multiplication by $k^i$.
 \item If $s < 0$ then for $i=1,\ldots,-ps-1$ the $C_p$-action on \[\un{H}_{t+2r-2i-1}(S^V; \un{\bZ})(G/C_q) \cong \bZ/q\] is given by multiplication by $k^{-i}$.
\end{enumerate}

\end{thm}

Now we can put this together to get a complete description of the Mackey-functor valued homology groups of a point:

\begin{thm} \label{t:Ghomologyofpoint}
Let $G$ be the nonabelian group of order $pq$ for $p \mid q-1$, and let $V = t \bR \oplus \bigoplus\limits_{i=1}^{(p-1)/2} r_i V_i \oplus \bigoplus\limits_{j \in J} s_j W_j$. Let $r=\sum r_i$ and $s=\sum s_j$. Then the Mackey functor valued homology groups of $S^V$ with $\un{\bZ}$-coefficients are the direct sum of the following Mackey functors:

\begin{enumerate}
\item \[
 \un{H}_{t+2r+2ps}(S^V; \un{\bZ}) \cong \begin{cases}
\un{\bZ} \quad & \textnormal{if $r+(p-1)s \geq 0$ and $s \geq 0$} \\
\un{\bZ}^{1,q} & \textnormal{if $r+(p-1)s \geq 0$ and $s < 0$} \\
\un{\bZ}^{p,1} & \textnormal{if $r+(p-1)s < 0$ and $s \geq 0$} \\
\un{\bZ}^{p,q} & \textnormal{if $r+(p-1)s < 0$ and $s < 0$} \\
\end{cases}
\]

\item If $r+(p-1)s > 0$, a $\wh{\bZ/p}$ in dimension $t+2s+2i$ for $i=0,\ldots,r+(p-1)s-1$.

\item If $r+(p-1)s < 0$, a $\wh{\bZ/p}$ in dimension $t+2s-2i-1$ for $i=1,\ldots,-(r+(p-1)s+1)$.

\item If $s > 0$, a $\wh{\bZ/q(k^i)}$ in dimension $t+2r+2i$ for $i=0,\ldots,ps-1$.

\item If $s < 0$, a $\wh{\bZ/q(k^{-i})}$ in dimension $t+2r-2i-1$ for $i=1,\ldots,-(ps+1)$.

\end{enumerate}

\end{thm}

\begin{example}
Let $p=3$ and $q=7$, and take $k=2$ in the presentation of $G$. Consider the virtual representation $V = - 4 + 7 V_1 - 2W_1$ of virtual dimension $-2$. To the eyes of $C_p$ this looks like $-8+3V_1$, and to the eyes of $C_q$ this looks like $10-6V_1'$. We conclude that the Mackey-functor valued homology group of $S^V$ with coefficients in $\un{\bZ}$ are as follows:

\begin{center}
\begin{tabular}{c|c|c|c|c|c|c|c|c|c}
 i & -8 & -6 & -4 & -2 & -1 & 1 & 3 & 5 & 7 \\
 \hline
 \rule{0pt}{4ex} $\un{H}_i(S^V; \un{\bZ})$ & $\wh{\bZ/p}$ & $\wh{\bZ/p}$ & $\wh{\bZ/p}$ & $\un{\bZ}^{1,q}$ & $\!\wh{\bZ/q}(2)$\! & \!$\wh{\bZ/q}(4)$\! & $\wh{\bZ/q}$ & \!$\wh{\bZ/q}(2)$\! & \!$\wh{\bZ/q}(4)$\! \\
 \hline
 \rule{0pt}{3ex} \!\!$H_i(S^V; \un{\bZ})(G/G)$\!\! & $\bZ/p$ & $\bZ/p$ & $\bZ/p$ & $\bZ$ & 0 & 0 & $\bZ/q$ & 0 & 0 %Hacked to get rid of overfull hbox
\end{tabular}
\end{center}

\end{example}

\section{The abelian group of order $pq$} \label{s:abgroupoforderpq}
In this section we let $G = C_{pq} \cong C_p \times C_q$ be the abelian group of order $pq$ for odd primes $p < q$. (The condition $p \mid q-1$ is no longer necessary.)

The representation theory of $G$ is much easier, with an irreducible $2$-dimensional representation $V_i$ for each $i=1,\ldots,\frac{pq-1}{2}$. Given a virtual representation $V = t \bR \oplus \bigoplus\limits_{i=1}^{(pq-1)/2} r_i V_i$, the homology calculation for $R^G_{C_p}(V)$ behaves like that of $t' \bR \oplus r V_1$ where $t' = t + 2 \sum\limits_{i \equiv 0 \mod p} r_i$ and $r = \sum\limits_{i \not \equiv 0 \mod p} r_i$. Similarly, the homology calculation for $R^G_{C_q}(V)$ behaves like that of $t'' \bR \oplus s V_1$ where $t'' = t + 2 \sum\limits_{i \equiv 0 \mod q} r_i$ and $s = \sum\limits_{i \not \equiv 0 \mod q} r_i$. By choosing $V$ appropriately we see that any tuple $(t', r, t'', s)$ with $t' \equiv t'' \mod 2$ is possible.

Unlike the calculation for the non-abelian group, there are no interesting group actions here and we get the following:

\begin{thm}
With notation as above, the Mackey functor valued homology groups of $S^V$ with $\un{\bZ}$-coefficients are the direct sum of the following Mackey functors:
\begin{enumerate}
\item \[
 \un{H}_{\dim_\bR(V)}(S^V; \un{\bZ}) \cong \begin{cases}
\un{\bZ} \quad & \textnormal{if $r \geq 0$ and $s \geq 0$} \\
\un{\bZ}^{1,q} & \textnormal{if $r \geq 0$ and $s < 0$} \\
\un{\bZ}^{p,1} & \textnormal{if $r < 0$ and $s \geq 0$} \\
\un{\bZ}^{p,q} & \textnormal{if $r < 0$ and $s < 0$} \\
\end{cases}
\]

\item If $r > 0$, a $\wh{\bZ/p}$ in dimension $t'+2i$ for $i=0,\ldots,r-1$.

\item If $r < 0$, a $\wh{\bZ/p}$ in dimension $t'-2i-1$ for $i=1,\ldots,-(r+1)$.

\item If $s > 0$, a $\wh{\bZ/q}$ in dimension $t''+2i$ for $i=0,\ldots,s-1$.

\item If $s < 0$, a $\wh{\bZ/q}$ in dimension $t''-2i-1$ for $i=1,\ldots,-(s+1)$.
\end{enumerate}
\end{thm}

\section{The alternating group $A_4$} \label{s:A4}
The description of an $A_4$-Mackey functor is similar to that of a $G$-Mackey functor as discussed in the previous sections, except that we have an additional level of subgroups. Let $K < A_4$ denote the Klein four subgroup. We will ignore the $3$ subgroups between $e$ and $K$ and draw an $A_4$-Mackey functor as
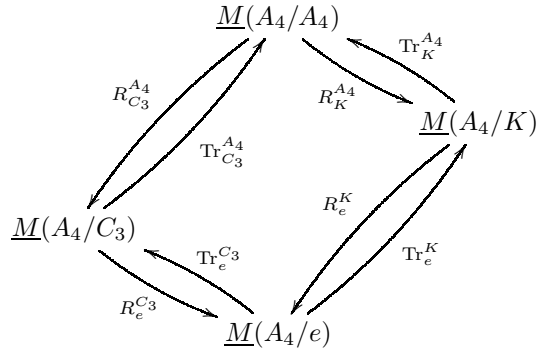
\[
 \xymatrix{ & \un{M}(A_4/A_4) \ar@/_/[rd]_-{R^{A_4}_K} \ar@/_/[ldd]_-{R^{A_4}_{C_3}} & \\ & & \un{M}(A_4/K) \ar@/_/[ul]_-{\tr_K^{A_4}} \ar@/_/[ldd]_-{R^K_e} \\ \un{M}(A_4/C_3) \ar@/_/[uur]_-{\tr_{C_3}^{A_4}} \ar@/_/[rd]_-{R^{C_3}_e} & & \\ & \un{M}(A_4/e) \ar@/_/[ruu]_-{\tr_e^K} \ar@/_/[lu]_-{\tr_e^{C_3}} & }
\]
We do not mean to imply that those subgroups are not important, but the difficulty is orthogonal to the discussion in this paper.

A calculation of $\un{H}_*(S^V; \un{\bZ})$ for $G=K$ can be found \cite[Section 4.8]{EB20}. In particular Ellis-Bloor finds that the Bockstein spectral sequence computing $\un{H}_*(S^V; \un{\bZ})(K/K)$ collapses at $E_2$ and that except for a single $\bZ$ in degree $\dim_\bR(V)$, $\un{H}_n(S^V; \un{\bZ})(K/K)$ is a direct sum of $\bZ/2$'s.

For $G=A_4$, the $C_3$-action on $\un{H}_n(S^V; \un{\bZ})(A_4/K)$ is given by cyclically permuting the $x_i$ and the $y_i$, so it can (with some difficulty) be read off from Ellis-Bloor's calculation. Except in degree $\dim_\bR(V)$, $\un{H}_n(S^V; \un{\bZ})(A_4/K)$ is a direct sum of copies of $\bZ/2$ with a trivial $C_3$-action, and $\bZ/2 \oplus \bZ/2$ where the generator of $C_3$ acts by $\left( \begin{smallmatrix} 0 & 1 \\ 1 & 1 \end{smallmatrix} \right)$. If we write the former as $\bF_2$ and the latter as $\bF_4$ we then get Mackey functors
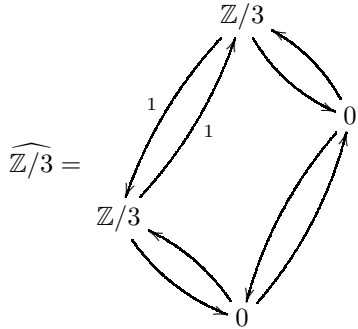
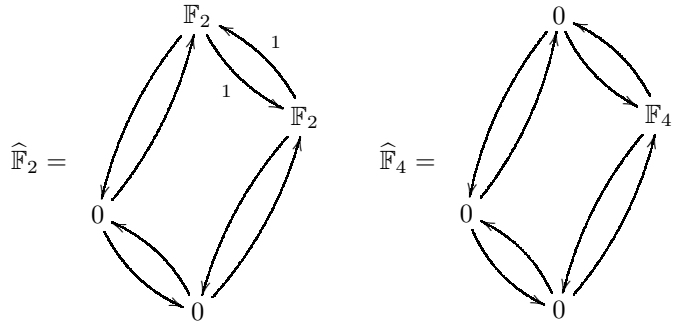
\[
 \wh{\bZ/3} = \vcenter{\hbox{\xymatrix{ & \bZ/3 \ar@/_/[rd] \ar@/_/[ldd]_-{1} & \\ & & 0 \ar@/_/[ul] \ar@/_/[ldd] \\ \bZ/3 \ar@/_/[uur]_-{1} \ar@/_/[rd] & & \\ & 0 \ar@/_/[ruu] \ar@/_/[lu] & } }}
\]
and
\[
 \wh{\bF}_2 = \vcenter{\hbox{ \xymatrix{ & \bF_2 \ar@/_/[rd]_-{1} \ar@/_/[ldd] & \\ & & \bF_2 \ar@/_/[ul]_-{1} \ar@/_/[ldd] \\ 0 \ar@/_/[uur] \ar@/_/[rd] & & \\ & 0 \ar@/_/[ruu] \ar@/_/[lu] & } }}
\qquad
 \wh{\bF}_4 = \vcenter{\hbox{ \xymatrix{ & 0 \ar@/_/[rd] \ar@/_/[ldd] & \\ & & \bF_4 \ar@/_/[ul] \ar@/_/[ldd] \\ 0 \ar@/_/[uur] \ar@/_/[rd] & & \\ & 0 \ar@/_/[ruu] \ar@/_/[lu] & } }}
\]
Here the restriction of $\wh{\bF}_2$ to $K$ can be any of the following three Mackey functors:
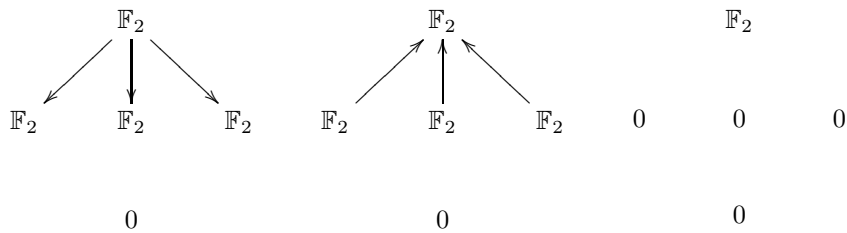
\[
 \xymatrix{
   & \bF_2 \ar[rd] \ar[d] \ar[ld] & \\
   \bF_2 & \bF_2 & \bF_2 \\
   & 0 &
 }
 \qquad
 \xymatrix{
   & \bF_2 & \\
   \bF_2 \ar[ru] & \bF_2 \ar[u] & \bF_2 \ar[lu] \\
   & 0 &
 }
 \qquad
 \xymatrix{
  & \bF_2 & \\
  0 & 0 & 0 \\
  & 0 &
 }
\]
The restriction of $\wh{\bF}_4$ to $K$ can be any of the following three Mackey functors:
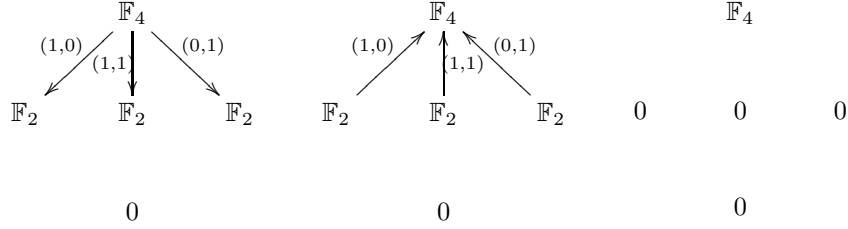
\[
 \xymatrix{
  & \bF_4 \ar[ld]_-{(1,0)\!\!\!} \ar[d]_{(1,1)\!\!\!} \ar[rd]^-{\!\!\!(0,1)} & \\
  \bF_2 & \bF_2 & \bF_2 \\
  & 0 &
 }
 \qquad
 \xymatrix{
  & \bF_4 & \\
  \bF_2 \ar[ru]^-{(1,0)\!\!\!} & \bF_2 \ar[u]_{\!\!\!(1,1)} & \bF_2 \ar[lu]_-{\!\!\!(0,1)} \\
  & 0 &
 }
 \qquad
 \xymatrix{
  & \bF_4 & \\
  0 & 0 & 0 \\
  & 0 &
 }
\]

The irreducible real representations of $A_4$ are $1$, $V_1$ and the standard representation $A$. We have $R^{A_4}_{C_3}(A) = 1 + V_1$ and $R^{A_4}_K(A) = \sigma_1 + \sigma_2 + \sigma_3$.

Given a virtual representation $V = t\bR \oplus r V_1 \oplus s A$, we can once again analyse $\un{H}_*(S^V; \un{\bZ})$. The number of $\wh{\bZ/3}$'s now depends on $r+s$, and the number of $\wh{\bF}_2$'s and $\wh{\bF}_4$'s depends (in some complicated way) on $s$. Finally, we have a single $\un{\bZ}$, $\un{\bZ}^{3,1}$, $\un{\bZ}^{1,4}$ or $\un{\bZ}^{3,4}$ in degree $\dim_\bR(V)$:
\[
 \un{H}_{t+2r+3s}(S^V; \un{\bZ}) \cong \begin{cases}
 \un{\bZ} \quad & \textnormal{if $r+s \geq 0$ and $s \geq 0$} \\
 \un{\bZ}^{1,4} & \textnormal{if $r+s \geq 0$ and $s < 0$} \\
 \un{\bZ}^{3,1} & \textnormal{if $r+s < 0$ and $s \geq 0$} \\
 \un{\bZ}^{3,4} & \textnormal{if $r+s < 0$ and $s < 0$} \\
 \end{cases}
\]

% \bibliographystyle{plain}
% \bibliography{b.bib}

\end{document}